\documentclass[a4paper,16pt,,reqno]{amsart}
\usepackage{amsmath, amsthm, amscd, amsfonts, amssymb, graphicx, color}
\usepackage[bookmarksnumbered, colorlinks, plainpages]{hyperref}
\usepackage{wrapfig}
\usepackage{subfig}
\usepackage{subfig}
\usepackage{cite}
\usepackage[framemethod=default]{mdframed}
\usepackage{showexpl}

\begin{document}
	 %%%%%%%%%%%%%%%%%%%%%%%%%%%%%%%%%%%%5
	
	 %\flushbottom
	 %%%%%%%%%%%%%%%%%%%%%%%%%%%%%%%%%%%%%%%%%%%%%
	 \newcommand{\be}{\begin{equation}}
	 \newcommand{\ee}{\end{equation}}
	 \newcommand{\bt}{\beta}
	 \newcommand{\al}{\alpha}
	 \newcommand{\laa}{\lambda_\alpha}
	 \newcommand{\lab}{\lambda_\beta}
	 \newcommand{\no}{|\Omega|}
	 \newcommand{\nd}{|D|}
	 \newcommand{\Om}{\Omega}
	 \newcommand{\h}{H^1_0(\Omega)}
	 \newcommand{\lt}{L^2(\Omega)}
	 \newcommand{\la}{\lambda}
	 \newcommand{\ro}{\varrho}
	 \newcommand{\cd}{\chi_{D}}
	 \newcommand{\cdc}{\chi_{D^c}}
	 \newtheorem{thm}{Theorem}[section]
	 \newtheorem{cor}[thm]{Corollary}
	 \newtheorem{lem}[thm]{Lemma}
	 \newtheorem{prop}[thm]{Proposition}
	 \theoremstyle{definition}
	 \newtheorem{defn}{Definition}[section]
	 \newtheorem{exam}{Example}[section]
     \newtheorem{remark}{Remark}
	 \theoremstyle{remark}
	 \newtheorem{rem}{Remark}[section]
	 \numberwithin{equation}{section}
	 \renewcommand{\theequation}{\thesection.\arabic{equation}}
	 \numberwithin{equation}{section}
	 %
	 %
	 %
	 %-------%
	 % TITLE %
	 %-------%
	 %------------------------------------------%
	 %------------------------------------------%
	\title[Uniqueness  result for long range spatially  segregation elliptic   system ]{Uniqueness  result for long range spatially  segregation elliptic   system }
	\author[Bozorgnia]{Farid Bozorgnia }
	 \address{Department of Mathematics, Instituto Superior T\'{e}cnico, Lisbon.} \email{bozorg@math.ist.utl.pt}

	% \date{\today}
	 %\date{October 31, 2015}

	 \thanks{ 
   F. Bozorgnia  was  partially   supported by the UT Austin-Portugal partnership through the FCT post-doctoral fellowship
	 	SFRH/BPD/33962/2009}
	
	 \begin{abstract}
	 	We study    a class of elliptic  competition-diffusion
systems of long range segregation models  for two and more competing species.  We  prove the  uniqueness result  for  positive solution  of  those elliptic and  related  parabolic systems when the coupling  in the right hand side involves  a   non-local  term  of integral form.

 Moreover,    alternate proofs of some  known
results, such as existence of solutions in the elliptic case and  the limiting configuration are given.   The  free boundary condition in a particular setting is given.

	 \end{abstract}

	 \maketitle
	 %------------------------------------------%

\section{Introduction and problem setting}

 One of the important problems  in population ecology is modeling of  competition and  interactions between biological components.
 To achieve this aim, different models based on reaction-diffusion equations are studied.  For  spatial segregation, two  following  models  have  been studied:

\begin{itemize}
\item   adjacent segregation:  in this model particles interact  on contact, and there is  a common  curve or hyper-surface of separation; free boundary;
\item   segregation at distance: species  interact  at a distance from each other. In this model,
the annihilation of the  coefficient for one component   at the point $x$     involves    the  values of  the rest of components  in a full neighborhood  of the point $x.$
\end{itemize}

 The adjacent segregation model and strongly competing systems  have  been extensively studied  from different point  of views,  we will explain briefly these perspective in the coming section,    see   \cite{CKL, CL, CTV1,  D,DHM,EY, K,KZ,Q} and references therein.  The model describes  the steady state   of $m$ competing species coexisting in the same area $\Omega.$
Let $u_{i}(x)$ denote the population density of the $i^{\textrm{th}}$ component with the internal dynamic prescribed by
$f_{i}(x,u_i).$  Then, the interaction between components is described by  the  following system of $m$
differential equations

\begin{equation}\label{s1}
\left \{
\begin{array}{lll}
-\Delta  u_{i}^{\varepsilon}= f_{i}(x,  u_{i}^{\varepsilon} ) - \frac{ 1 }{\varepsilon}  u_{i}^{\varepsilon} \sum\limits_{j \neq i}   (u_{j}^{\varepsilon})^{\beta} (x) & \text{ in  } \Omega,\\
u_{i}(x) =\phi_{i}(x)    &   \text{ on   } \partial \Omega,\\
i=1,\cdots, m.
 \end{array}
\right.
\end{equation}
Here $\phi_{i}$ are non-negative  $C^{1,\alpha}$  functions with disjoint supports that is, $\phi_{i} \cdot \phi_{j}  =0,$
 on  the boundary. In the  system (\ref{s1}), the parameter   $\beta$  can be chosen   $\beta=1$ or $2$ which for the case $\beta=2,$ the system is in variational form.

 To explain the second   model, first  we indicate  some of  the notations that  we are dealing with  in this paper.

\begin{itemize}
\item $\Omega \subset \mathbb{R}^d,$  is  bounded  domain with  $C^{1, \alpha}$    boundary;
   \item  $d(x,\partial \Omega)$ denotes the distance of the point $x$ to  $\partial \Omega;$
 \item for a given $ D  \subset \mathbb{R}^d,$ we define  $(D)_1:={\{  x \in \mathbb{R}^d : \,  d(x,  D)\le 1}\};$
 \item $ (\partial \Omega)_1:={\{  x \in \Omega ^{c}: \,  d(x,\partial \Omega)\le 1}\},$
%\item $ (\partial \Omega)_1:={\{  x \in \Omega ^{c}: \,  d(x,\partial \Omega) = 1}\},$
\item  supp f :   the support of function $f;$
  \item $B_{r}(x)={\{ y \in \mathbb{R}^n: |x-y|<  r }\};$
  \item   $W^{+}=\max(W,0)$  and   $W^{-}=\max(-W,0).$
 % \item  $\chi_{A}(x):$  the characteristic function of a given set $A$.
  % \item For a set  $ A \subset \mathbb{R}^n, \, B_{r}(\partial A):=  \underset{ x \in \partial A} {\bigcup} B_{r}(x).$
\end{itemize}
% % % % % % % % % % % % % % % % % % % % % % % % % % % %
 In this work,  we  consider the following elliptic   system studied in \cite{CPQ}:

\begin{equation}\label{s2}
\left \{
\begin{array}{lllll}
\Delta  u_{i}^{\varepsilon}=\frac{ 1 }{\varepsilon}  u_{i}^{\varepsilon} \sum\limits_{j \neq i } H( u_{j}^{\varepsilon})(x) & \text{ in  } \Omega,\\
u_{i}^{\varepsilon}(x) =\phi_{i}(x)    &   \text{ on   } (\partial \Omega)_1,\\
 \end{array}
\right.
\end{equation}
where
 \begin{equation}\label{nnew1}
H(  u_{j}^{\varepsilon})(x)=\int_{B_{1}(x)} u_{j}^{\varepsilon}(y) dy,
\end{equation}
or
\begin{equation}\label{nnew2}
  H(  u_{j}^{\varepsilon})(x)= \underset{ y\in B_{1}(x)}{\sup}  u(y).
\end{equation}
Here,  the boundary data  $\phi_{i}$  for $i=1,\cdots ,m$ are non-negative, $C^{1, \alpha}$  functions  defined
on $ (\partial \Omega)_1$  with supports at distance, at least  one   from each other, i.e,
\[
(\text{supp} \, \phi_{i}  )_{1} \cap ( \text{supp} \, \phi_{j} )^{\circ}= \emptyset,\quad  \text{for} \, i\neq j.
\]
System (\ref{s2}) can also  be viewed as steady state of   the  following parabolic system
\begin{equation}\label{s3}
\left \{
\begin{array}{lll}
\frac{\partial u_{i}^{\varepsilon}}{\partial t}  -\Delta  u_{i}^{\varepsilon}=  - \frac{ 1 }{\varepsilon}
 u_{i}^{\varepsilon} \sum\limits_{j \neq i } H(  u_{j}^{\varepsilon})(x) & \text{ in  }
 Q:=\Omega\times(0,+\infty),\\
u_{i}^{\varepsilon}(x,t) =\phi_{i}(x,t)    &   \text{ on   } (\partial \Omega)_1,\\
u_{i}^{\varepsilon}(x,0) =u_{i,0}   & \text{in } \Omega \times
{\{t=0}\}.
 \end{array}
\right.
\end{equation}

The main contribution of this work is to provide  uniqueness
results  for system  \eqref{s2}  (Lemma 3.2)  and  system \eqref{s3} when $H$ is given by (\ref{nnew1}). Moreover,  we  provide alternate proof of known results, such as existence of solutions in the elliptic case with right hand side   given by (\ref{nnew1}).  Also  we show that as the competition rate goes to infinity the solution converges, along with suitable sequences,
to a spatially  long range segregated state (Lemma 4.5),   more deep results about properties of  limiting configuration can be found in \cite{CPQ}.
%We point out that existence of a positive solution
%to system (\ref{s2}) is  already proven in \cite{2} and  also  the uniqueness result is obtained by using the
%method introduced in \cite{14}.

The outline of this paper is as follows.   In  Section 2  we  provide mathematical background and known
    results about  the systems  \eqref{s1} and  \eqref{s2}.  Section 3 deals with existence and uniqueness for systems  (\ref{s2}) and    (\ref{s3}) where $H$ is given by (\ref{nnew1}).  Section 4  consists  analysis of the system   (\ref{s2})  in the limiting case as $\varepsilon$ tends to zero when $m=2$.

\section{Basic facts}

In this section we review some  of  known results and  mathematical background  for two systems (\ref{s1}) and (\ref{s2}). The analysis of  the system (\ref{s2}) is much more difficult compare with system (\ref{s1}). Understanding the properties of the system (\ref{s1}) gives some insights in the  study of system  (\ref{s2}).  Roughly speaking,    the system (\ref{s2})  can be reduced to  the system (\ref{s1}) if in the term
  \[
H(  u_{j}^{\varepsilon})(x)=\underset{ y\in B_{1}(x)}{\sup}  u(y),
\]
  instead of unit ball,  we   consider the ball with radius $r$  where $r$ tends to zero, then
  \[
      H(  u_{j}^{\varepsilon})(x)=  u_{j}^{\varepsilon}(x)
      \]
       so it is possible to see some general behaviour in the system  (\ref{s2}).
\subsection{Known results  for  the first model}
  As  we already mentioned, the  system (\ref{s1}) has been studied well. First,  the existence of solution for each $\varepsilon$  is shown in \cite{CTV1,K,KZ} i.e,     for each $ \varepsilon $ the   system (\ref{s1}) admits a solution $(u_{1}^{\varepsilon},...,u_{m}^{\varepsilon}) \in (H^{1}(\Omega))^{m}.$ Moreover,  it is  also shown that  for each $\varepsilon $  the normal derivative of  $ u_{i}^{\varepsilon} $  is bounded independent of $\varepsilon $ which implies that  there exists  $(u_{1},...,u_{m}) \in (H^{1}(\Omega))^{m}$ such that up to subsequences,   we have the strong convergence of   $ u_{i}^{\varepsilon}  $ to $u_{i} $ in   $H^{1}(\Omega),$  and $u_{i}\cdot u_{j}=0$ for $i\neq j.$ For fixed $ \varepsilon $ uniqueness of elliptic system  (\ref{s1}) for $ f_{i}  \equiv  0$,  $ \beta=1$     and parabolic system have
  been shown in  \cite{KZ}.  In \cite{CTV2} for a class of segregation state governed by a variational principle, existence of solutions is shown and also the conditions that  provide the uniqueness are given. To see uniqueness result for limiting case when $ \varepsilon$ tends to zero, see \cite{AB,KZ}.
   We refer to \cite{FB}  to  see numerical approximation of the system (\ref{s1})  for the  limiting case as  $\varepsilon $ tends to zero.
  %The long time  dynamics for a two   components competing and diffusing system with  inhomogeneous Dirichlet boundary condition is given in \cite{6,7,8,13}.

     Another observed  result in \cite{CTV1},  is that for regular  points  on the interface separating the support of  $u_{i} $ and $u_{j} $  the following holds
\begin{equation}\label{sun1}
\underset {u_{i}(y) > 0} {\underset{  y\rightarrow x} { \lim}} \, \nabla u_{i}(y)=-  \underset {u_{j}(y) > 0} {\underset{ y\rightarrow x} { \lim}} \,  \nabla u_{j}(y).
\end{equation}
%
%The properties of solution and free boundary for class  variational and non variational models   in  (\ref{s1}) as $ \varepsilon$ tends to zero    has been widely studied.
The limiting  solutions of (\ref{s1})  share the following properties  and belong to class $S$ in below,     \cite{CTV1}
\begin{multline*}
\textit{S}=\left\{U=(u_{1},...,u_{m})\in H^{1}(\Omega):u_{i}\geq0,u_{i} \cdot  u_{j}=0  \text{ if }   i\neq j, \right.\\
\left. u_{i}=\phi_{i}  \text{ on }    \partial \Omega , \, -\triangle u_{i}\leq 0, \, -\Delta\left(u_{i}-\sum _{j\neq i}u_{j}\right)\geq 0 \right\}.
\end{multline*}

\begin{remark}
 In  system  (\ref{s1})  when $\varepsilon \rightarrow 0 $  the system in variational  form,  i.e.,  $\alpha =2$ has same solution as  the system with $ \alpha=1.$
 \end{remark}

In the case of two components i.e., $m=2$   the explicit solution can be obtained as following.  Note that in this case the difference of two functions,  $u_{1}^{\varepsilon}-u_{2}^{\varepsilon}, $ is harmonic for each   $\varepsilon.$  Let $W$  be the harmonic extension on $\Omega$ of the boundary data $\phi_{1}-\phi_{2} $. If  we set $u_{1}=W^{+},$  $ u_{2}=W^{-}$, then  the pair $(u_{1},u_{2})$ is the limit configuration of any sequences of pairs $(u_{1}^{\varepsilon},\, u_{2}^{\varepsilon})$, and  there  exists $C\geq0$ such that, (see \cite{CTV1})
\begin{equation}\label{sunshin1}
 (\frac{1}{\varepsilon})^{1/6}.\parallel u_{i}^{\varepsilon}-u_{i}\parallel_{H^{1}_{0} (\Omega)}\leq C \ \text{ as } \ \varepsilon\rightarrow 0.
\end{equation}

 Recently in  \cite{STT}, the   regularity issues for system of strongly competing Schro\"{o}dinger equation  with nontrivial grouping has been studied. The   $C^{0, \alpha} $ estimate that are uniform in competition parameter, also the regularity of free boundary as competition rate tends to infinity, are obtained, we refer to \cite{NTTV} for more  related work.

 %Regularity of the nodal set of segregated critical
%configurations for a  class of Lipschitz vector functions  which its  components are nonnegative, disjointly supported and verify an elliptic equation on each  support been studied in \cite{9}.  The authors   prove that the nodal set is a collection of  $C^{1,\alpha}$  hyper-surfaces. They applied the result  to the asymptotic limits of reaction-diffusion systems with strong competition interactions, to optimal partition problems involving eigenvalue.

\subsection{Long  range segregated model}
Now,  we turn  our attention  to the second system given by (\ref{s2}). System  (\ref{s2}) is  in variational form if
\[
H(  u_{j}^{\varepsilon})(x)=\int_{B_{1}(x)}( u_{j}^{\varepsilon}(y))^2 dy.
\]
    \begin{remark}
  In system  (\ref{s1}),  the interaction between components is given by  the term $ u_{i}(x)  u_{j}(x);$ while in  (\ref{s2})
    components interacting by the nonlocal   term  $ u_{i}^{\varepsilon}  H( u_{j}^{\varepsilon})(x).$
    The analysis  and asymptotic behaviour of the   system  (\ref{s1}) are  more straightforward  than  system (\ref{s2}). For instance,
      if the number of components  $m=2$,  then   in   system (\ref{s1}) with $f_{i} =0, \beta=1$,  the difference $  u_{1}^{\varepsilon}-u_{2}^{\varepsilon}$ is harmonic for each  $\varepsilon$ while this is not true for  system  (\ref{s2}).
    \end{remark}
    %
%Assumptions on the boundary data $\phi_{i}(x),$  functions $f_i$   are as follows:
%\begin{itemize}
%\item The boundary data  $\phi_{i}$  for $i=1,\cdots ,m$ are non-negative, $C^{1}$ continuous functions defined
%on $ (\partial \Omega)_1$  with supports at distance at least  one from each other, i.e,
%\[
%(\text{supp} \, \phi_{i}  )_{1} \cap ( \text{supp} \, \phi_{j} )_{1}= \emptyset,\quad  \text{for} \, i\neq j.
%\]
% \item  Functions  $ f_i$ for $i=1,\cdots ,m$     are $C^1$ on $[0, \quad  \infty) $ with
% \[
% f_{i}(0)=0,   \quad  f(u)<0  \text{ for all  } \, u>1.\\
% \]
%\end{itemize}
In  \cite{CPQ}   rigorous analysis  is  done to show the following :
\begin{itemize}
\item There exist continuous functions $  u_{1}^{\varepsilon}, \cdots ,u_{m}^{\varepsilon} $  depending on the parameter  $ \varepsilon $ which solve the system  (\ref{s2}) in  viscosity sense.

\item
 As  $\varepsilon$ tends to zero, there  exists a subsequence  $u _{i}^{\varepsilon_{k}} $  converging locally uniformly,  to a function
$u_{i}$, satisfying the  properties that  the $ u_{i}$'s are
 locally Lipschitz continuous in $\Omega$
 and have supports at distance at least one  from each other.

\item  Each function $u_{i}$ is harmonic on its support. The authors show the semi convexity of the free boundary. For the points belonging  to free boundary,   there is an exterior tangent ball of radius one  at $x_0.$

 \item  The free boundary set has finite $(n - 1)$-dimensional Hausdorff  measure and free boundary
set is a set of finite perimeter.

\item They obtained sharp characterization of the interfaces, i.e,
the supports of the limit functions are at distance exactly  one from each other.

\item  Free boundary condition in any dimension for two components  is given  when $H$ is defined by  (\ref{nnew1}).

\end{itemize}

\section{Existence and uniqueness of the nonlocal segregation model}

Consider the following elliptic system
\begin{equation}\label{sy1}
\left \{
\begin{array}{lllll}
\Delta  u_{i}^{\varepsilon}=  \frac{ 1 }{\varepsilon}  u_{i}^{\varepsilon} \sum\limits_{j \neq i} \int_{B_{1}(x)} u_{j}^{\varepsilon}(y)\, dy    & \text{ in  } \Omega,\\
u_{i}(x) =\phi_{i}(x)    &   \text{ on   } (\partial \Omega)_1.\\
 \end{array}
\right.
\end{equation}
Existence of the solution for system  (\ref{sy1})   has been shown in
\cite{CPQ} by Schauder fixed point argument. The aim of this work is to cover the  lack of uniqueness  for solution  of (\ref{sy1}).  We  show  uniqueness   of  solution  for system (\ref{sy1})  inspired  by the proof of uniqueness for system (\ref{s1})  in   \cite{KZ}.  Since the proof is constructive
it  can be used for numerical simulation to approximate the
solution  of  $\varepsilon$ problem in (\ref{sy1}).
 \begin{lem}\label{sys1}
For each $\varepsilon >0, $ there exists a  positive solution
 $(u_{1}^{\varepsilon},\cdots  ,u_{m}^{\varepsilon})$ of System  (\ref{sy1}).

 \end{lem}
 \begin{proof}
 To start, consider the harmonic extension $u_{i}^{0}$ given by
 \begin{equation}\label{sys2}
 \left \{
 \begin{array}{llll}
   \Delta u_{i}^{0} = 0   & \text{ in  } \Omega,\\
   u_{i}^{0}  =\phi_{i}     &   \text{on   } (\Omega)_1\setminus\Omega.
  \end{array}
 \right.
 \end{equation}
%We  also set
 %\[ u_{i}^{0} = \phi_{i} \quad   \text{in }  \quad   (\partial \Omega)_1. \]
  Now, given $ u_{i}^{k}$ consider the solution of the following linear   system
 \begin{equation}\label{sy4}
 \left \{
 \begin{array}{lllll}
 \Delta  u_{i}^{k+1}=  \frac{ 1 }{\varepsilon}  u_{i}^{k+1} \sum\limits_{j \neq i} H(  u_{j}^{k} ) (x) & \text{ in  } \Omega,\\
 u_{i}^{k+1}(x) =\phi_{i}(x)    &   \text{on   } (\Omega)_1\setminus\Omega.\\
  \end{array}
 \right.
 \end{equation}
 We show that the following inequalities hold:
 \[
 u_{i}^{0}\ge u_{i}^{2} \cdots \ge u_{i}^{2k}\ge \dots  \ge u_{i}^{2k+1}\ge \cdots \ge u^{3}_{i}\ge u_{i}^{1},  \quad \textrm{in} \, \Omega.
 \]
 Note that since $u_{i}^{0} \ge 0$ then
 \[
  \sum_{j \neq i} \int_{B_{1}(x)}   u_{j}^{0}(y) dy \ge 0,  \quad  x \in \Omega.
  \]
 The   boundary conditions  $\phi_{i}(x)$  are non negative so the weak maximum principle   implies  that $ u_{i}^{1}\ge 0$ and consequently
 \[
  u_{i}^{k}\ge 0, \quad \text{ for} \,  k\ge1, i=1, \cdots ,m.
  \]
 Now we have
  \begin{equation}\label{sy44}
  \left \{
  \begin{array}{ll}
  \Delta  u_{i}^{1}\ge 0 & \text{ in  } \Omega,\\
  u_{i}^{1}(x) =u_{i}^{0}(x)= \phi_{i}(x)    &   \text{ on   } \partial \Omega.\\
   \end{array}
  \right.
  \end{equation}
Thus  the comparison principle  implies that $ u_{i}^{1}\le  u_{i}^{0} $.  To proceed more with induction, assume that

  \begin{equation}\label{inq2}
 u_{i}^{0}\ge u_{i}^{2}\ge  \cdots \ge u_{i}^{2k}\ge  u_{i}^{2k+1}\ge \cdots \ge u^{3}_{i}\ge u_{i}^{1}.
  \end{equation}
 We show that
 \begin{equation}\label{2.1}
 u_{i}^{2k+2}\ge u_{i}^{2k+1}.
 \end{equation}
  By (\ref{sy4}) and the  assumption in  (\ref{inq2})  we have
  \[
  \Delta  u_{i}^{2k+2}\le  \frac{ 1 }{\varepsilon}  u_{i}^{2k+2} \sum_{j \neq i} H(  u_{j}^{2k} ) (x),
  \]
 \[
 \Delta  u_{i}^{2k+1}=  \frac{ 1 }{\varepsilon}  u_{i}^{2k+1} \sum_{j \neq i} H(  u_{j}^{2k} ) (x).
 \]
 Note that  $  u_{i}^{2k+1}  $ and $  u_{i}^{2k+2}$ have the same boundary value so \eqref{2.1} follows from the comparison
 principle.   The same argument  using the assumption  $u_{i}^{2k+1} \ge u_{i}^{2k-1}$   shows that
   \[
    u_{i}^{2k+2}\le u_{i}^{2k}.
    \]
   For the next step, we note that
  \begin{equation}\label{syss2}
  \left \{
  \begin{array}{llll}
   \Delta u_{i}^{2k+3} =  \frac{ 1 }{\varepsilon}  u_{i}^{2k+3} \sum\limits_{j \neq i} H(  u_{j}^{2k+2} ) (x)    & \text{ in  } \Omega,\\
   \Delta u_{i}^{2k+1} =  \frac{ 1 }{\varepsilon}  u_{i}^{2k+1} \sum\limits_{j \neq i} H(  u_{j}^{2k} ) (x)    & \text{ in  } \Omega.\\
       \end{array}
  \right.
  \end{equation}
 From previous step  we have $  u_{i}^{2k+2}\le u_{i}^{2k}$ which implies
 \[
    u_{i}^{2k+3}  \ge  u_{i}^{2k+1}.
    \]
   Now let $ \overline{u}_{i} $ and  $ \underline{u}_{i} $ be two families of functions such that
   \[
    u_{i}^{2k} \rightarrow   \overline{u}_{i} \quad  \textrm{uniformly in } \Omega,
    \]
     \[
        u_{i}^{2k+1} \rightarrow \underline{u}_{i} \quad  \textrm{uniformly in } \Omega.
        \]
Taking the limit in   (\ref{sy4}) yields

  \begin{equation}\label{syss3}
  \left \{
  \begin{array}{llll}
  \Delta \overline{u}_{i} =  \frac{ 1 }{\varepsilon}  \overline{u}_{i} \sum\limits_{j \neq i} H( \underline{u}_{j} ) (x)    & \text{ in  } \Omega,\\
 \Delta \underline{u}_{i} =  \frac{ 1 }{\varepsilon}  \underline{u}_{i} \sum\limits_{j \neq i} H( \overline{u}_{j} ) (x)    & \text{ in  } \Omega.
       \end{array}
  \right.
  \end{equation}
 The inequality $   u_{i}^{2k+1}\le u_{i}^{2k}$ implies that
 \begin{equation}\label{in1}
  \overline{u}_{i} \ge \underline{u}_{i}\quad \text{in} \quad \Omega.
 \end{equation}
   We will show that,  in fact,  the equality holds.  Since  $\overline{u}_{i}=\underline{u}_{i}$ on $\partial\Omega$, by \eqref{in1}, we have
  \begin{equation}\label{syss60}
     \frac{\partial  \overline{u}_{i} } {\partial n} \le \frac{\partial  \underline{u}_{i} } {\partial
     n},
     \end{equation}
where $n$ is the outward  normal vector of $\partial\Omega$. Hence
 \begin{equation}\label{in180}
\int_{\Omega}\sum_{i} \Delta \overline{u}_{i}(x)\, dx= \int_{\partial \Omega}\sum_{i}  \frac{\partial  \overline{u}_{i} } {\partial n} \, ds \le \int_{\partial \Omega}  \sum_{i}\frac{\partial  \underline{u}_{i} } {\partial n} \, ds= \int_{\Omega}\sum_{i} \Delta \underline{u}_{i}(x)\, dx.
 \end{equation}
Substituting  equation (\ref{syss3})   into (\ref{in180}), we obtain
 \begin{equation}\label{syss61}
\int_{\Omega} \sum_{{i,j}_{ j \neq i}}   \overline{u}_{i}(x) \left(
\int_{B_{1}(x)}  \underline{u}_{j}(y)\, dy \right) \, dx \le
\int_{\Omega} \sum_{{i,j}_{j \neq i}}     \underline{u}_{i}(x) \left(
\int_{B_{1}(x)} \overline{u}_{j}(y)\, dy\right) \, dx.
 \end{equation}
Rewriting this,  we get a symmetric kernel $K(x, y);$  such that
 \begin{equation}\label{syss62}
\int_{\Omega} \int_{\Omega_1} \sum_{{i,j}_{j \neq i}}   \overline{u}_{i}(x)   \underline{u}_{j}(y)  K(x,y) \, dy dx  \le \int_{\Omega} \int_{\Omega_1}  \sum_{{i,j}_{j \neq i}}     \underline{u}_{j}(x)   \overline{u}_{i}(y) K(x,y) \, dy \, dx,
 \end{equation}
where $K(x,y)$  is    $\chi_{B_1(0)}(x-y)$ with $\chi_{B_1(0)}$
the characteristic function of the unit ball centered at the origin.  Since  $K $ is symmetric in $x$  and $y,$
 \begin{equation}\label{syss63}
\int_{\Omega} \int_{\Omega} \sum_{{i,j}_{j \neq i}}   \overline{u}_{i}(x)   \underline{u}_{j}(y)  K(x,y) \, dy dx  = \int_{\Omega} \int_{\Omega}
 \sum_{{i,j}_{j \neq i}}     \underline{u}_{j}(x)   \overline{u}_{i}(y) K(x,y) \, dy \, dx.
 \end{equation}
The remaining part is
 \begin{equation}\label{syss64}
\int_{\Omega} \int_{\Omega_{1}\backslash \Omega} \sum_{{i,j}_{j \neq i}}   \overline{u}_{i}(x)   \underline{u}_{j}(y)  K(x,y) \, dy dx  = \int_{\Omega} \int_{\Omega_{1}\backslash \Omega}  \sum_{{i,j}_{j \neq i}}     \overline{u}_{i}(x)\phi_{j}(y)    K(x,y) \, dy \, dx
 \end{equation}
\[
\ge \int_{\Omega} \int_{\Omega_{1}\backslash \Omega} \sum_{{i,j}_{j \neq i}} \underline{u}_{i}(x) \phi_{j}(y)   K(x,y) \, dy dx  = \int_{\Omega} \int_{\Omega_{1}\backslash \Omega}  \sum_{{i,j}_{j \neq i}}    \underline{u}_{j}(x)   \overline{u}_{i}(y)    K(x,y) \, dy \, dx.
\]
Combining (\ref{syss62})-(\ref{syss64})  we obtain
\begin{equation}\label{new1}
\int_\Omega\int_{\Omega_1\setminus\Omega}\sum_{{i,j}_{j \neq i}}\overline{u}_{i}(x)\phi_j(y)K(x,y)dydx=\int_\Omega\int_{\Omega_1\setminus\Omega}\sum_{{i,j}_{j \neq i}} \underline{u}_{i}(x)\phi_j(y)K(x,y)\,dy\,dx.
\end{equation}
Now from  (\ref{new1})  we obtain
\[
  \overline{u}_{i}(x)= \underline{u}_{i}(x) \quad \text{in} {\{ x\in \Omega : dist(x, \partial \Omega)\le 1}\}.
  \]
  This  follows from facts that $\overline{u}_{i} \ge   \underline{u}_{i}$ and non negativity of boundary data  $ \phi_{j} $  and   definition of kernel $K(x,y).$  In  view of  (\ref{syss63}) and the  continuation argument we obtain
\[
 \overline{u}_{i}\equiv  \underline{u}_{i}, \quad \text{ in }
 \Omega,
\]
which is a solution of \eqref{sy1}.
\end{proof}

\begin{lem}\label{sysn1}(Uniqueness)
Assume there exists  another  positive solution $(w_1,\cdots,w_m)$ of  (\ref{sy1}), then
\[
u_{i}=w_{i}.
\]
\end{lem}
\begin{proof}
We will prove that the following hold:
 \begin{equation}\label{ineq1}
u_{i}^{2k+1} \le w_{i}\le u_{i}^{2k}, \quad  \textrm{ for } \,    k \ge 0.
 \end{equation}
To begin, we show that
 \begin{equation}\label{ineq20}
w_{i}\le u_{i}^{0}.
\end{equation}
This is a consequence of the fact that   $w_{i}$ satisfies
 \begin{equation*}
  \left \{
  \begin{array}{llll}
   \Delta w_{i}\ge 0    & \text{ in  } \Omega,\\
     w_{i}=u^{0}_{i}   & \text{on   } \partial \Omega.
       \end{array}
  \right.
  \end{equation*}
Next we compare $w_{i}$  with $u_{i}^1$ and we show $ w_{i} \ge u_{i}^1$.  This  inequality follows from (\ref{ineq20}) and

 \begin{equation*}
  \left \{
  \begin{array}{ll}
   \Delta w_{i} =\frac{w_i}{\varepsilon}  \sum\limits_{j \neq i}  \int_{B_{1}(x)} w_j   & \text{ in  } \Omega,\\
    \Delta u_{i}^1   =\frac{u_{i}^{1}}{\varepsilon}  \sum\limits_{j \neq i}  \int_{B_{1}(x)} u_{j}^{0}   & \text{ in  } \Omega.
       \end{array}
  \right.
  \end{equation*}
%We  use  the  previous result states  $ w_j \le u_{j}^{0}$.
 Now we proceed by induction and we assume that the claim is true until $2k+1.$  This means that we have
\[
u_{i}^{2k+1}\le w_{i} \le u_{i}^{2k}.
\]
Then we show
\[
u_{i}^{2k+3}\le w_{i} \le u_{i}^{2k+2}.
\]
Again we can compare the equations in below
  \begin{equation*}
   \left \{
   \begin{array}{ll}
    \Delta w_{i} =\frac{w_i}{\varepsilon}  \sum\limits_{j \neq i} \int_{B_{1}(x)}w_j   & \text{ in  } \Omega,\\
     \Delta u_{i}^{2k+2}   =\frac{u_{i}^{2k+2}}{\varepsilon}  \sum\limits_{j \neq i} \int_{B_{1}(x)}u_{j}^{2k+1}   & \text{ in  } \Omega.
        \end{array}
   \right.
   \end{equation*}
Here we use that $  u_{j}^{2k+1} \le w_j$ which implies that $ w_{i} \le u_{i}^{2k+2}.$   Also we have
 \begin{equation*}
   \left \{
   \begin{array}{ll}
    \Delta w_{i} =\frac{w_i}{\varepsilon}  \sum\limits_{j \neq i} \int_{B_{1}(x)}w_j   & \text{ in  } \Omega,\\
     \Delta u_{i}^{2k+3}   =\frac{u_{i}^{2k+2}}{\varepsilon}  \sum\limits_{j \neq i} \int_{B_{1}(x)}u_{j}^{2k+2}   & \text{ in  } \Omega.
        \end{array}
   \right.
   \end{equation*}
By the last step $ u_{i}^{2k+2}\ge  w_{i},$ which implies
\[
   u_{i}^{2k+3}\le   w_{i}.
   \]
  Now taking limit in (\ref{ineq1}) shows that
  \[
  w_i=u_i.\qedhere
  \]
\end{proof}
As a corollary of Lemmas  \ref{sys1}  and  \ref{sysn1}  we have the following theorem.
\begin{thm}

For each $\varepsilon >0, $ there exists a unique positive solution
 $ ( u_{1}^{\varepsilon},\cdots ,  u_{m}^{\varepsilon})$ of System  (\ref{sy1}).
\end{thm}

The same method can be used to construct the unique solution to the
parabolic problem \eqref{s3}.
Indeed, we can proceed as before to construct functions
$\overline{u}_i\geq\underline{u}_i$, which satisfy
\begin{equation*}
  \left \{
  \begin{array}{llll}
  \Delta \overline{u}_{i}-\frac{\partial \overline{u}_{i}}{\partial t} =  \frac{ 1 }{\varepsilon}  \overline{u}_{i} \sum\limits_{j \neq i} H( \underline{u}_{j} ) (x)    & \text{ in  } \Omega,\\
 \Delta \underline{u}_{i} -\frac{\partial \underline{u}_{i}}{\partial t}=  \frac{ 1 }{\varepsilon}  \underline{u}_{i} \sum\limits_{j \neq i} H( \overline{u}_{j} ) (x)    & \text{ in  } \Omega.
       \end{array}
  \right.
  \end{equation*}
Similarly, we still have
\begin{equation*}
     \frac{\partial  \overline{u}_{i} } {\partial n} \le \frac{\partial  \underline{u}_{i} } {\partial
     n},\quad{\mbox{on} }\quad \partial\Omega.
     \end{equation*}
Hence for any $T>0$,
\[\int_0^T\int_{\partial\Omega}\frac{\partial  \overline{u}_{i} } {\partial n} \le \int_0^T\int_{\partial\Omega}\frac{\partial  \underline{u}_{i} } {\partial
     n}.\]
Substituting the equation into this, the left hand side equals
\begin{eqnarray*}
&&\int_0^T\int_{\Omega}\frac{\partial  \overline{u}_{i} } {\partial
t}+\frac{ 1 }{\varepsilon}  \overline{u}_{i} \sum\limits_{j \neq i} H(
\underline{u}_{j} ) \\
&=&\int_0^T\int_{\Omega}\frac{ 1 }{\varepsilon}  \overline{u}_{i}
\sum\limits_{j \neq i} H( \underline{u}_{j}
)+\int_\Omega\overline{u}_i(x,T)dx-\int_\Omega u_{i,0}(x)dx,
\end{eqnarray*}
and a similar one holds for the right hand side. By noting that
\[\int_\Omega\overline{u}_i(x,T)dx\geq\int_\Omega\underline{u}_i(x,T)dx,\]
we obtain
\[\int_0^T\int_{\Omega}\frac{ 1 }{\varepsilon}  \overline{u}_{i}
\sum_{j \neq i} H( \underline{u}_{j}
)\leq\int_0^T\int_{\Omega}\frac{ 1 }{\varepsilon}  \underline{u}_{i}
\sum_{j \neq i} H( \overline{u}_{j} ).\] The rest of the  proof is
exactly the same as before.

\section{Basic estimates and asymptotic behavior as $\varepsilon$ tends to zero}
In this part we  study the  elliptic systems  with highly  competitive interaction term.  We provide the estimates for the case that competition rate tends to infinity which yields  the long range distance of positive components. Although the  complete analysis   and  more results of limiting case can be found in  \cite{CPQ},  here we  simplify some proofs.

For simplicity,  we assume that  the number of components is $m=2$ and we consider the following system
\begin{equation}\label{system7}
\left \{
\begin{array}{lllll}
\Delta  u^{\varepsilon}=  \frac{ u^{\varepsilon} }{\varepsilon} \int_{B_{1}(x)} v^{\varepsilon}(y) dy  & \text{ in  } \Omega ,\\
\Delta  v^{\varepsilon} = \frac{ v^{\varepsilon}}{\varepsilon} \int_{B_{1}(x)}   u^{\varepsilon}(y) dy  & \text{ in  } \Omega, \\
u^{\varepsilon}(x) =\phi(x)    &   \text{ on   } (\partial \Omega)_1,\\
v^{\varepsilon}(x) =\varphi(x)    &   \text{ on    } (\partial \Omega)_1.\\
 \end{array}
\right.
\end{equation}
%Here we assume that $f, g \in C^{1}[0, \, \infty)$ such that
%$f(0)=g(0)=0$ and for all $s>1 $ we  have $f(s)< 0, g(s)<0.$
%Let $u^{\varepsilon} $ and $v^{\varepsilon}$ be the positive  solution of system     (\ref{system7}). Then $u^{\varepsilon} $ and $v^{\varepsilon}$ are bounded.
% %\[
%%  0\le  u^{\varepsilon}(x) \le \max\phi, \quad  0\le v^{\varepsilon}(x) \le \max\varphi.
%% \]
%%
%%Positivity of   $u^{\varepsilon} $  and  $v^{\varepsilon} $ follows from the facts that the function $ 0 $ satisfies the equations in (\ref{system7}) and $u, v$ are positive on the boundary.
%To show the inequalities  $u^{\varepsilon} \le \max\phi $
%and  $v^{\varepsilon} \le  \max\varphi$ the   is done
%by contradiction. Assume the maximum of $u^{\varepsilon} $ is
%attained at an $x_{0} \in \Omega $ such that
%$u^{\varepsilon}(x_{0})>M.$ Then positivity of
%integral of   $v^{\varepsilon} $ in  $B_{1}(x_{0}) $ in  the first equation of (\ref{system7})
%shows:
%\[
% \Delta u^{\varepsilon}(x_{0}) =\frac{ u^{\varepsilon}(x_{0}) }{\varepsilon} \int_{B_{1}(x_{0})} v^{\varepsilon}(y) dy  > 0,
% \]
%  which is impossible.
 We use the next Lemma in \cite{CPQ} which states in a strip of size one around the support of a component on the boundary the other components decays to zero exponentially.
 \begin{lem}\label{c1}
 For $\sigma>0,$ let
 \[
 \overline{\Gamma}^{\sigma} :={\{ \phi(x)> \sigma}\} \subset \Omega^{c}.
 \]
 Then on the set ${\{ x \in \Omega: d(x, \overline{\Gamma}^{\sigma}) \le 1-r}\}, \quad 0< r< 1 $, we have
 \[
 v^{\varepsilon} \le C e^{\frac {- c \sigma^{\alpha} r^{\beta} } {\sqrt{\varepsilon}}}.
 \]
  \end{lem}
\begin{lem}\label{faridb2}
 Assume that the boundary $ \partial \Omega$  satisfies  an  uniform exteririor ball  condition.    Let $( u^{\varepsilon},   v^{\varepsilon}) $ be the  positive  solution of  (\ref{system7}). There exists a positive constant $C$ independent of  $\varepsilon$  such that
\[
 \underset{x\in \partial \Omega}{\sup} |\frac{\partial  u^{\varepsilon}(x) } {\partial n} |\le  C,
\]
\[
 \underset{x\in \partial \Omega}{\sup} |\frac{\partial  v^{\varepsilon}(x) }{\partial n} |\le  C,
\]
where $n$ denotes exterior normal to $\partial \Omega.$
 \end{lem}
 \begin{proof}
 We construct  barrier functions to control the bound of gradient of  $u^{\varepsilon}$ and  $v^{\varepsilon}$ as follows. Firstly,  the following  inequalities hold
\[-\Delta u^\varepsilon\leq 0,\quad -\Delta v^\varepsilon\leq 0.\]
 By the standard sup-sub solution method, we can construct solutions
$\overline{u}$ and $\overline{v}$ to the problem
\begin{equation*}
\left \{
\begin{array}{llll}
  \Delta \overline{u} =0    & \text{ in  } \Omega,\\
 \Delta \overline{v} = 0 & \text{ in  } \Omega,\\
\overline{u} =\phi   &   \text{on   }  \partial \Omega,\\
\overline{v} =\varphi   &   \text{on }  \partial \Omega.
 \end{array}
\right.
\end{equation*}
Moreover,
\[u^{\varepsilon}\leq\overline{u}, \quad v^{\varepsilon}\leq\overline{v}.\]
Hence
\begin{equation}\label{001}
\frac{\partial u^{\varepsilon}}{\partial n}\geq \frac{\partial
\overline{u}}{\partial n},\quad \frac{\partial v^\varepsilon}{\partial n}\geq
\frac{\partial \overline{v}}{\partial n}.
\end{equation}
Note that such $\overline{u}$ and $\overline{v}$ are independent of
$\varepsilon$. At the part where $\phi=0$, because $u\geq 0$ in $\Omega$, we also
have
\[\frac{\partial u^\varepsilon}{\partial n}\leq 0.\]
Combined with \eqref{001}, we get a uniform bound on $\frac{\partial
u^\varepsilon}{\partial n}$.  It remains to consider the case on $\{\phi>0\}$. Take an
$x_0\in\partial\Omega$ such that $\phi(x_0)>0$. By the previous
lemma,
\[v^\varepsilon(x)\leq Ce^{-\frac{1}{C\sqrt{\varepsilon}}},\quad\mbox{in } B_{\frac{1}{2}}(x_0),\]
where $C$ depends on $\phi(x_0)$. Then in $\Omega\cap B_{\frac{1}{2}}(x_0)$, $u^\varepsilon$ satisfies
\[
\Delta u^\varepsilon\leq \frac{C}{\varepsilon}e^{-\frac{1}{C\sqrt{\varepsilon}}}u^\varepsilon.\]
From this we can construct a solution $w^{\varepsilon}$ to the problem
\begin{equation*}
\left \{
\begin{array}{llll}
  \Delta w^{\varepsilon} =\frac{C}{\varepsilon}e^{-\frac{1}{C\sqrt{\varepsilon}}}w^\varepsilon    & \text{ in  } \Omega\cap B_{\frac{1}{2}}(x_0),\\
 w^{\varepsilon} =   u^\varepsilon   &   \text{on   }  \partial (\Omega\cap
B_{\frac{1}{2}}(x_0)).
 \end{array}
\right.
\end{equation*}
Moreover,
\[u^\varepsilon\geq w^\varepsilon,\quad\mbox{in}\ \Omega\cap B_{\frac{1}{2}}(x_0).\]
Hence
\[\frac{\partial u^\varepsilon}{\partial n}\leq \frac{\partial w^{\varepsilon}}{\partial n},\quad\mbox{on }\partial\Omega\cap B_{\frac{1}{2}}(x_0).\]
Note that
\[
0 < \frac{C}{\varepsilon}e^{-\frac{1}{C\sqrt{\varepsilon}}}\le K,
\]
 where $K$  is a constant independent of $ \varepsilon$.  By standard boundary gradient estimates, there exists a constant
$C>0$ independent of $\varepsilon$, such that
\[\frac{\partial
w^\varepsilon}{\partial n}\leq C,\quad\mbox{on }\partial\Omega\cap
B_{\frac{1}{4}}(x_0).\] Take a finite cover of
$\partial\Omega\cap\{\phi>0\}$ using balls $B_{\frac{1}{4}}(x_i)$ with
$x_i\in\partial\Omega\cap\{\phi>0\}$, we see
\[\frac{\partial
u^\varepsilon}{\partial n}\leq C,\quad\mbox{in
}\partial\Omega\cap\{\phi>0\}.\]
 Combining this with \eqref{001} we get a uniform bound on $\frac{\partial u^\varepsilon}{\partial n}$
in the part $\partial\Omega\cap\{\phi>0\}$.
\end{proof}
\begin{lem}\label{faridb3}
There exist a constant $C$ independent of $\varepsilon$ such that if $ ( u^{\varepsilon}, v^{\varepsilon})$ is  a solution of
system (\ref{system7})   then

\[
 \int_{\Omega} \frac{ u^{\varepsilon} }{\varepsilon} (\int_{B_{1}(x)} v^{\varepsilon}(y) \, dy) \, dx \le C,
\]
\[
 \int_{\Omega} \frac{ v^{\varepsilon} }{\varepsilon} (\int_{B_{1}(x)} u^{\varepsilon}(y) \, dy) \, dx \le C.
\]
\end{lem}
\begin{proof}
By integrating of the first equation in (\ref{system7}) over $\Omega$,   we have
\[ \int_{\Omega} \frac{ u^{\varepsilon} }{\varepsilon} (\int_{B_{1}(x)} v^{\varepsilon}(y) \, dy) \, dx =  \int_{\Omega} \Delta  u^{\varepsilon}\, dx
=\int_{\partial \Omega} \frac{  \partial u^{\varepsilon}}{\partial n} \, ds.
\]
Now  Lemma  \ref{faridb2} give the result.
\end{proof}

\begin{lem}\label{faridb4}
 There exists a positive constant $C_2 $ independent of  $\varepsilon$  such that
 \[
   \int_{\Omega}   |  \nabla u^{\varepsilon}|^{2} \,   dx \le C_2,
 \]
\[
  \int_{\Omega}   |  \nabla v^{\varepsilon}|^{2}  \, dx  \le C_2.
 \]

\end{lem}

\begin{proof}
We multiply the differential inequality $ -\Delta u_{\varepsilon} \le 0$    by   $u^{\varepsilon}$ and integration over $\Omega$ gives
\[
\int_{\Omega}   |  \nabla u^{\varepsilon}|^{2} \,   dx - \int_{ \partial \Omega} u^{\varepsilon} \frac{\partial  u^{\varepsilon} }{\partial n} ds\le
0.
\]
Now the bound in gradient in    Lemma \ref{faridb2} give the result.
\end{proof}

\begin{lem}\label{s11}
Let $u$ and $v$ be the limiting solution of (\ref{system7}).  Assume that $x_{0}$  is  a point in $\Omega$ such that $ u(x_{0})>0$. Then we have
\[
  v \equiv 0  \quad \textrm{in} \quad  B_{1}(x_{0}).
\]
\end{lem}

\begin{proof}
By Lemma \ref{faridb3} we have
\[
 \int_{\Omega}  u^{\varepsilon}(x)  (\int_{B_{1}(x) } v^{\varepsilon}(y) \, dy) \, dx \le C\varepsilon.
\]
Let  $\varepsilon $ tends  to zero in the above inequality  to get
\[
0 \le  \int_{\Omega}  u (x)  (\int_{B_{1}(x)}  v(y) \, dy) \, dx \le 0.
\]
This implies
\[
 \int_{\Omega}  u (x)  (\int_{B_{1}(x)}   v(y) \, dy) \, dx =0,
 \]
 which shows
 \[
  v\equiv 0  \quad \textrm{in} \quad  B_{1}(x_{0}).\qedhere
\]
\end{proof}
\begin{remark}
Let $u$ and $v$ be the limiting solution of (\ref{system7}) as $ \varepsilon $ tends to zero. Lemma (\ref{s11}) shows that the support of $u$ and the support of $v$  are  disjoint at distance at least one. In fact in \cite{CPQ} it is shown that $u$ and $v$  are  exactly at  distance one.
\end{remark}

\begin{defn}
The boundaries $ \partial {\{x\in \Omega: u(x)>0}\},  \partial{\{x\in \Omega: v(x)>0}\} $ are called free boundaries.
\end{defn}

\subsection{Free boundary condition in dimension one}
In \cite{CPQ} for any dimension, the free boundary condition for limiting solution is given for
\[
H(  u_{j}^{\varepsilon})(x)=\int_{B_{1}(x)} u_{j}^{\varepsilon}(y) dy,
\]
 The following simple argument gives the free boundary condition in dimension one when $H$ is given by (\ref{nnew2}).  Let $d=1,$ $\Omega=(-a, a),$ and  $ a \ge 1,$ consider the following system
 \begin{equation}\label{27}
\left \{
\begin{array}{lllll}
 ( u^{\varepsilon}(x))^{\prime\prime}=   \frac{ u^{\varepsilon}(x) }{\varepsilon} \underset{y \in [x-1,x+1]}{\sup} v^{\varepsilon}(y)   & \text{ in  } (-a,\,  a) ,\\
 ( v^{\varepsilon}(x))^{\prime\prime} =  \frac{v^{\varepsilon}(x)}{\varepsilon}\underset{y \in [x-1,x+1] }{\sup}
u^{\varepsilon}(y)    & \text{ in  } (-a, \, a), \\
u^{\varepsilon}, v^{\varepsilon}(y)\ge 0 & \text{ in  } (-a, \, a), \\
u^{\varepsilon}(x) =\phi(x)    &   \text{ on  } [-a-1,\,  -a],\\
v^{\varepsilon}(x) =\varphi(x)    &   \text{ on   } [a,\,  a+1].\\
 \end{array}
\right.
\end{equation}

It is easy to see that
\[
\underset{y \in [x-1,x+1]}{\sup} v^{\varepsilon}(y)=v^{\varepsilon}(x+1).
\]
Also we have that
\[
( v^{\varepsilon}(x+1))^{\prime\prime} =  \frac{v^{\varepsilon}(x+1)}{\varepsilon}\underset{y \in [x, x+2]}{\sup}
u^{\varepsilon}(y) =\frac{v^{\varepsilon}(x+1)}{\varepsilon} u^{\varepsilon}(x).
\]
This shows for every $\varepsilon,$
\begin{equation}\label{22}
( u^{\varepsilon}(x)- v^{\varepsilon}(x+1))^{\prime\prime}=0.
\end{equation}
 Let $ u, v$ be the limiting points as  $\varepsilon $ tends to zero.  Then $u$ and $v$ satisfy the following system
\begin{equation}\label{2}
\left \{
\begin{array}{lllll}
( u(x)-v(x+1)  )^{\prime\prime }= 0  & \text{ in  } (-a, a) ,\\ ( v(x)-u(x-1) )^{\prime\prime} = 0  & \text{ in  } (-a,a), \\
  u, v \ge 0   &   \text{ in   } (-a, a),\\
  u(-a) =\phi(-a)   \quad v(a) =\varphi(a).\\
 \end{array}
\right.
\end{equation}
This shows that in (\ref{2})  if  $x_f$ be a free boundary point then the following holds,  compare with (\ref{sun1}).
\[
u^{\prime}(x_f)= -v^{\prime}(x_{f}+1).
\]

\section{Conclusion and further works}
  The   uniqueness of the solution  for   a class of elliptic  competition-diffusion
systems of long range segregation models   is   shown.  Also we show  as the competition rate goes to infinity,  the solution converges  to a spatially  long range segregated state satisfying some  free boundary problems.

   In a forthcoming paper the author will present numerical approximation  for   the  class of elliptic and parabolic  competition-diffusion
systems of long range segregation models  for two and more competing species.
%\section{Appendix}
% Weak Maximum principle:
%Consider elliptic operator $L$ having non divergence form
%\[
%Lu =- \sum_{i,j=1}^{n} a_{ij} u_{x_{i}x_{j}} + \sum_{i}^{n} b_{i} u_{x_{i}} + cu,
%\]
%where the coefficients  $a_{ij}, b_{i}, c $ are continuous.  Assume $ u\in C^{2}(\Omega) \cap  C(\overline{\Omega})$ and $c\ge 0$ in $\Omega$.
%\begin{itemize}
%\item
%If
%\[
%L u \le 0, \quad \text{ in } \, \Omega,
%\]
%then
%\[
%\underset{\overline{\Omega}}{\max u} \le  \underset{\partial \Omega}{\max u^+}
%\]
%\item
%Likewise, if
%\[
%L u \ge 0, \quad \text{ in } \, \Omega,
%\]
%then
%\[
%\underset{\overline{\Omega}}{\min u} \le  \underset{\partial \Omega}{\min u^-}
%\]
%
%\end{itemize}
\section*{Acknowledgments}

The author  has  great thanks to Prof. Caffarelli for suggesting problem and  also  would like to thank to Kelei. Wang for  helpful
 suggestions and discussion.

\renewcommand{\refname}{REFERENCES }


\begin{thebibliography}{99}\label{Refs}


\bibitem{AB}   A. Arakelyan, F. Bozorgnia,  {\it On  the uniqueness  of the  limiting solution to a strongly competing system}. Electronic Journal of Differential Equations.{\bf 96}  (2017), 1-8.\\

\bibitem{FB} F. Bozorgnia, {\it Numerical algorithms for the spatial segregation of competitive systems.} SIAM J. Sci.
Comput.  No. 5   (2009),   3946-3958.\\


\bibitem{CPQ}   L. Caffarelli, S. Patrizi and V. Quitalo, {\it On a long range segregation  model.}   J. Eur. Math. Soc, (2017).\\


\bibitem{CKL}       L. Caffarelli,  A. L. Karakhanyan,  and F.-H. Lin. {\it The geometry of solutions to a segregation problem for nondivergence systems}.  Journal of Fixed Point Theory and Applications. {\bf 5}   (2009),  319-351.\\

\bibitem{CL}      L. Caffarelli and F.-H. Lin.  {\it  Singularly perturbed elliptic systems and multi-valued harmonic functions with free boundaries.}  Journal of the American Mathematical Society. {\bf 21}  (2008),  847-862.\\





\bibitem{CTV1} M. Conti, S. Terracini, and G. Verzini, {\it Asymptotic estimate for spatial segregation of competitive  systems.} Advances in Mathematics. {\bf 195}  (2005),     524-560.\\


\bibitem{CTV2} M. Conti, S. Terracini, and G. Verzini, {\it  A varational problem for the spatial segregation of reaction-diffusion systems}, Indiana Univ. Math. J. {\bf 54},  no.3     (2005),  779--815.\\



\bibitem{D} E.N. Dancer, {\it Competing species systems with diffusion and large interaction}, Rend. Sem. Mat.Fis. Milano. {\bf 65}   (1995),   23-33.\\

\bibitem{DHM} E.N. Dancer, D. Hilhorst, M. Mimura, and  L.A. Peletier, {\it Spatial segregation limit of a competition
diffusion system,}  European J. Appl. Math. {\bf 10}    (1999),   97-115.\\





 \bibitem{NTTV} B. Noris,  H.  Tavares, S.  Terracini,   and  G. Verzini, {\it     Uniform H\"{o}lder Bounds
for Nonlinear Schr\"{o}dinger Systems with Strong Competition},  Communications on Pure and Applied Mathematics,
 {\bf 63}  (2010),   267--302. \\

\bibitem{STT}    N.  Soave, H.  Tavares, S.  Terracini, A.  Zilio, {\it H\"{o}lder bounds and regularity of emerging free boundaries forstrongly competing Schr\"{o}dinger equations with nontrivial grouping,}  Nonlinear Analysis: Theory, Methods and   Applications  {\bf 138}  (2016), 388-427.\\


 \bibitem{EY}   S-I. Ei, E. Yanagida, {\it Dynamics of interfaces in competition-diffusion systems}, SIAM J. Appl.
Math. {\bf 54}   (1994),   1355-1373.\\


 \bibitem{K}  K.wang  {\it   Free  Boundary  Problems  and  Asymptotic  Behavior  of  Singularly
Perturbed  Partial  Differential  Equations.}   Springer  Theses. Springer, 2013.\\


\bibitem{KZ}  K. Wang, Z.  Zhang,  {\it  Some new results in competing systems with many species},  Ann. I. H. Poincar\'{e}-- AN.  {\bf 27}  (2010),   739-761.\\



\bibitem{Q} V. Quitalo,  {\it A free boundary problem arising from segregation of populations with high competition.
Archive for Rational Mechanics and Analysis.}    {\bf 210} (2013),   857-908.




\end{thebibliography}
\end{document}